\documentclass[12pt, a4paper]{article}

\newtheorem{theorem}{Theorem}[section]

\newtheorem{lemma}[theorem]{Lemma}

\newtheorem{remark}{Remark}

\newenvironment {proof} {{\it Proof.}}{\hspace*{\fill}$\Box$\par\vspace{4mm}}

\usepackage{graphicx}
\usepackage{amssymb}
\usepackage{epsfig}
\usepackage{amssymb}
\usepackage{amsmath}
\usepackage{amsfonts}
\usepackage{graphicx, url}
\usepackage{makeidx}
\usepackage{cite}
\usepackage{fancyhdr}
\usepackage{color}
\usepackage{setspace}
\usepackage{amsmath}
\usepackage{amsfonts}

\newcommand{\mc}{\mathcal}
\newcommand{\mb}{\mathbb}

\begin{document}

\title{Stochastic maximum principle with Lagrange multipliers and optimal consumption with L\'{e}vy wage}


\author{K. R. Dahl\ 
\and E. Stokkereit 
}

\footnotetext{The research leading to these results has received funding from the European Research Council under the European Community's Seventh Framework Programme (FP7/2007-2013) / ERC grant agreement no. 228087.}


\maketitle

\begin{abstract}

We show how a stochastic version of the Lagrange multiplier method can be combined with the stochastic maximum principle for jump diffusions to solve certain constrained stochastic optimal control problems. Two different terminal constraints are considered; one constraint holds in expectation and the other almost surely.

As an application of this method, we study the effects of inflation- and wage risk on optimal consumption. To do this, we consider the optimal consumption problem for a budget constrained agent with a L{\'e}vy income process and stochastic inflation. The agent must choose a consumption path such that his wealth process satisfies the terminal constraint. We find expressions for the optimal consumption of the agent in the case of CRRA utility, and give an economic interpretation of the adjoint processes. 
\end{abstract}


\section{Introduction}
\label{sec: Intro}

This paper derives a stochastic Lagrange multiplier method, to solve constrained optimal control problems for jump diffusions. This can be used in combination with methods of optimal control, such as the stochastic maximum principle. Two different terminal constraints are considered, one that holds in expectation (soft constraint), and one that holds almost surely (hard constraint). Moreover, this method is used to analyze an interesting optimal consumption problem with wage jumps and stochastic inflation.

The problem of determining optimal lifetime consumption under uncertainty dates back to the seminal papers of Merton~\cite{Merton69}, \cite{Merton70}. These papers treat the problem of portfolio choice in continuous time in a complete market. In particular, there is no risk in the wage-level and the price of the consumption good is constant. This is explored further by Karatzas and Shreve~\cite{KaratzasShreve}, by using a so-called martingale method to handle market incompleteness. Karatzas~\cite{KaratzasEtAl} also considers a dynamic, stochastic economy with several heterogeneous agents. 


To analyze our version of the optimal consumption problem, we first impose a constraint on the expected terminal level of savings. This constraint transfers all the risk to the relevant financial institution (bank), and the consumers behave as if the market was complete. We assume that the agents have constant relative risk aversion (CRRA) utility functions and seek to maximize expected utility over a finite time horizon. Consequently, we are able to arrive at an explicit expression for an agent's optimal consumption process. Second, we impose an almost sure constraint on the terminal level of savings. This constraint is similar to the concept of admissibility widely used in the finance literature (see e.g. Karatzas and Shreve ~\cite{KS}), and makes the consumers bear all market risk. Thus, two extremes of risk sharing are considered.


A motivation for studying a consumption optimization problem with stochastic income and a CRRA utility function is found in Zeldes~\cite{Zeldes}. The paper argues that such uncertainties dramatically affects the consumption function, and links this to three classical empirical consumption puzzles. The model in Zeldes~\cite{Zeldes} uses discrete time and numerical approximations for the solution. In contrast, we consider continuous time and focus on analytical solutions.

A paper which, similarly to our paper, has a more analytical approach to such a consumption optimization problem is El Karoui and Jeanblanc-Picqu{\'e}~\cite{ElKarouiJeanblanc}. The authors solve the consumption-portfolio problem for an agent with a stochastic, insurable income under a liquidity constraint. Opposed to our situation, El Karoui and Jeanblanc-Picqu{\'e} assume that the wage is not a source of new uncertainty, and they have a liquidity constraint which prohibits all borrowing against future income.

Koo~\cite{Koo} also considers a consumption-portfolio problem for a liquidity constrained agent, but who has an uninsurable income risk. In this case, the wage process is modeled using only a Brownian motion. This is contrary to our situation, where the wage process is a jump process. Note that none of the articles mentioned so far specifically consider the inflation risk, as we do in this paper.

It is interesting to include inflation risk because having a stochastic future consumption price adds more realism to the model, yet it does not make the problem significantly more complicated. A paper which does consider the inflation risk, as well as the income risk, is Battocchio and Menoncin~\cite{BattocchioMenoncin}. However, their wage process does not include jumps (only a Brownian motion).


The structure of this paper is as follows. In Section~\ref{sec: StochLagrange} we introduce a general stochastic optimal control problem with constraints, and prove a stochastic Lagrange multiplier method for solving this type of problem. This is the theoretical and methodical foundation for the rest of the paper. Then, in Section~\ref{sec: Model}, we treat an interesting application from economy in detail: We introduce a specific stochastic control problem involving inflation- and wage risk. In Section~\ref{sec: Solution} the problem is solved using the stochastic multiplier method of Section~\ref{sec: StochLagrange}. We consider the case with a soft end constraint (i.e., the constraint holds in expectation), where an explicit solution is derived using the Lagrange multiplier method and the stochastic maximum principle. 



\section{A stochastic Lagrange multiplier method}
\label{sec: StochLagrange}

Let $(\Omega, \mc{F}, P)$ be a probability space, where $\Omega$ is the scenario space, $\mc{F}$ a $\sigma$-algebra and $P$ the probability measure. We consider continuous time, $t \in [0,T]$. Let $\{B(t)\}_{t \in [0,T]}$ be a Brownian motion, and $\int_{\mb{R}} z \tilde{N}(dt,dz)$ a pure jump process independent of $B(t)$. Let $\{\mc{F}_t\}_{t \in [0,T]}$ be the filtration generated by the Brownian motion and the pure jump process. In the following, by an adapted processes, we mean adapted with respect to this filtration.

Also, let $f : \mb{R}_+ \times \mb{R} \times \mb{R} \rightarrow \mb{R}$ and $g: \mb{R} \rightarrow \mb{R}$ be given, continuous functions. Then, we consider the stochastic optimal control problem which comes in two versions $(i)$ and $(ii)$:
\begin{equation}
 \label{eq: OpprProblem}
\begin{array}{lll}
 \sup_{u \in \mc{A}} &E^x[\int_0 ^T f(t,X(t),u(t)) dt + g(X(T))] \\[\smallskipamount]
\mbox{subject to} \\[\smallskipamount]
&dX(t) = b(t,X(t),u(t))dt + \sigma(t,X(t),u(t))dB(t) \\[\smallskipamount]
&\hspace{1.35cm} + \int_{\mb{R}} \gamma(t,X(t^-),u(t^-),z) \tilde{N}(dt,dz) \\[\smallskipamount]
&(i) \mbox{ } E^x[M(X(T))] = 0 \mbox{ or } (ii) \mbox{ } M(X(T)) = 0 \mbox{ a.s.},
\end{array}
\end{equation}
\noindent where $M : \mb{R} \rightarrow \mb{R}$ is some given continuous function, $\mc{U} \subset \mb{R}$ is a given set, $b: \mb{R}_+ \times \mb{R} \times \mc{U} \rightarrow \mb{R}$, $\sigma: \mb{R}_+ \times \mb{R} \times \mc{U} \rightarrow \mb{R}$ and $\gamma: \mb{R}_+ \times \mb{R} \times \mc{U} \times \mb{R} \rightarrow \mb{R}$ . Here, $E^x[\cdot]$ denotes the expectation given that the state process $(X(t))_{t \in [0,T]}$ starts in $x$, i.e. $X(0)=x$.

In problem~(\ref{eq: OpprProblem}), the stochastic process $u(t) = u(t,\omega): \mb{R}_+ \times \Omega \rightarrow \mc{U}$ is our control process. We say that this control process $u(t)$ is admissible, and write $u \in \mc{A}$ if the dynamics of $X$ (i.e., the SDE in problem~(\ref{eq: OpprProblem})) has a unique, strong solution for all $x \in \mb{R}$, and
\[
E^x[\int_0 ^T f(t,X(t),u(t)) dt + g(X(T))] < \infty.
\]
There is a sufficient stochastic maximum principle for jump diffusions by Framstad et al.~\cite{FOS} (see also Tang and Li~\cite{TangLi}), which can be used to find the optimal control of problem~(\ref{eq: OpprProblem}) without the additional constraints $(i)$ or $(ii)$. This maximum principle converts the simplified problem into the problem of maximizing a Hamiltonian function, and solving the so-called adjoint backward stochastic differential equation (BSDE). For reference purposes, this stochastic maximum principle, Theorem~\ref{thm: SufficientMaxPrinc}, is written out in Appendix~\ref{sec: Appendix}.

However, if we add a constraint such as $(i)$ or $(ii)$ to the problem, such as in (\ref{eq: OpprProblem}), the stochastic maximum principle cannot be used directly. We consider these two different kinds of constraints, and show how the constrained stochastic optimal control problems can be solved using a combination of a generalized Lagrange duality method and the stochastic maximum principle.

\begin{remark}
 \label{remark: Dimension}
For notational simplicity, problem~(\ref{eq: OpprProblem}) is assumed to be in one dimension. However, the results of this section also apply to multi-dimensional stochastic optimal control problems.
\end{remark}

\subsection{Constraint of type $(i)$}
\label{subsec: Expected}

Consider problem~(\ref{eq: OpprProblem}) with a type $(i)$ constraint, i.e.:
\[
\begin{array}{ll}
\phi(x) \mbox{ }:= \mbox{ } \sup_{u \in \mc{A}} \mbox{ } E^x[\int_0 ^T f(t,X(t),u(t)) dt + g(X(T))] \hspace{2.5cm}\\[\smallskipamount]
\mbox{subject to}
\end{array}
\]
\begin{equation}
\begin{array}{rllll}
 \label{eq: Problem1}
dX(t) &=&b(t,X(t),u(t))dt + \sigma(t,X(t),u(t))dB(t) \\[\smallskipamount]
&&+ \int_{\mb{R}} \gamma(t,X(t^-),u(t^-),z) \tilde{N}(dt,dz) \\[\smallskipamount]
E^x[M(X(T))] &=& 0.
\end{array}
\end{equation}

This problem can be solved using the standard Lagrange multiplier method, and then applying some method of stochastic control, for instance the stochastic maximum principle. Hence, let $\lambda \in \mb{R}$ be a Lagrange multiplier. Then, we introduce the unconstrained stochastic control problem
\begin{equation}
 \label{eq: Problem1Lagrange}
\begin{array}{llllll}
\phi_{\lambda}(x) \mbox{ } := \mbox{ } \sup_{u \in \mc{A}} \mbox{ } E^x[\int_0 ^T f(t,X(t),u(t)) dt + g(X(T)) + \lambda M(X(T))] \\[\smallskipamount]
\mbox{subject to} \\[\smallskipamount]
\hspace{1.6cm}dX(t) = b(t,X(t),u(t))dt + \sigma(t,X(t),u(t))dB(t) \\[\smallskipamount]
\hspace{3cm}+ \int_{\mb{R}} \gamma(t,X(t^-),u(t^-),z) \tilde{N}(dt,dz).
\end{array}
\end{equation}
This solution strategy is explored in Section 11.3 in {\O}ksendal~\cite{Oksendal} for the no-jump case. However, the proof of this theorem generalizes in a straight-forward manner to the L\'{e}vy case (i.e., where we may have jumps in the dynamics of the state process $X(t)$). Therefore, we have the following theorem.

\begin{theorem}
\label{thm: 1}
(Theorem 11.3.1 in {\O}ksendal~\cite{Oksendal}, adapted to the jump case)
Suppose that we for all $\lambda \in \mb{R}$ can find $\phi_{\lambda}(y)$ and $u_{\lambda}^*$ solving the unconstrained stochastic control problem~(\ref{eq: Problem1Lagrange}). Moreover, suppose there exists $\lambda_0 \in \mb{R}$ such that
\[
 E^x[M(X_{u_{\lambda_0}^*}(T))] =0.
\]
Then, $\phi(x) := \phi_{\lambda_0}(x)$ and $u^* := u^*_{\lambda_0}$ solves the constrained stochastic control problem~(\ref{eq: Problem1}).
\end{theorem}

\begin{proof}
 See {\O}ksendal~\cite{Oksendal}, Theorem 11.3.1 and also the proof of the following Theorem~\ref{thm: 2}.
\end{proof}

From Theorem~\ref{thm: 1}, it is sufficient to solve problem~(\ref{eq: Problem1Lagrange}) in order to solve problem~(\ref{eq: OpprProblem}). Problem~(\ref{eq: Problem1Lagrange}) can be solved using the stochastic maximum principle (see Theorem~\ref{thm: SufficientMaxPrinc} in the Appendix).

\subsection{Constraint of type $(ii)$}
\label{subsec: StokLagrLik}

Now, consider problem~(\ref{eq: OpprProblem}) with a type $(ii)$ constraint. That is, consider the stochastic optimal control problem

\[
\begin{array}{lll}
\phi(x) \mbox{ } := \mbox{ } \sup_{u \in \mc{A}} \mbox{ } E^x[\int_0 ^T f(t,X(t),u(t)) dt + g(X(T))] \hspace{2.5cm} \\[\smallskipamount]
\mbox{subject to} \\[\smallskipamount]
\end{array}
\]
\begin{equation}
 \label{eq: Problem2}
\begin{array}{lrllll}
&dX(t) &=&b(t,X(t),u(t))dt + \sigma(t,X(t),u(t))dB(t) \\[\smallskipamount]
&&&+ \int_{\mb{R}} \gamma(t,X(t^-),u(t^-),z) \tilde{N}(dt,dz) \\[\smallskipamount]
&M(X(T)) &=& 0 \mbox{ a.s.}
\end{array}
\end{equation}
\noindent where, as before, $M : \mb{R} \rightarrow \mb{R}$ is a given, continuous function. For notational simplicity, define
\[
J^u(x) := E^x[\int_0 ^T f(t,X(t),u(t)) dt + g(X(T))].
\]
We would like to use the Lagrange multiplier concept to solve problem~(\ref{eq: Problem2}) by solving an unconstrained stochastic control problem. However, since we have an almost sure constraint, it is not sufficient to introduce a single scalar Lagrange multiplier $\lambda \in \mb{R}$. The Lagrange multiplier must be stochastic in order to handle the stochastic constraint $M(X(T)) = 0 \mbox{ a.s.}$ Hence, we introduce an $\mc{F}_T$-measurable stochastic Lagrange multiplier $\mu : \Omega \rightarrow \mb{R}$ (which we will also call a \emph{stochastic multiplier}). Note that $\mu$ must be $\mc{F}_T$-measurable, since $M(X(T))$ is $\mc{F}_T$-measurable.

Assume that the stochastic multipler $\mu$ satisfies $E[\mu] < \infty$. Moreover, assume that $E^x[M(X_u(T)] < \infty$ for all $u \in \mc{A}$. We introduce a new, but related stochastic control problem
\begin{samepage}
\[
\begin{array}{ll}
\phi_{\mu}(x) \mbox{ } := \mbox{ }  \sup_{u \in \mc{A}} \mbox{ } E^x[\int_0 ^T f(t,X(t),u(t)) dt + g(X(T)) + \mu M(X(T))] \hspace{1cm} \\[\medskipamount]
\mbox{subject to}
\end{array}
\]
\begin{equation}
 \label{eq: LagrangeProblem2}
\begin{array}{lrllll}
dX(t) &=&b(t,X(t),u(t))dt + \sigma(t,X(t),u(t))dB(t) \\[\smallskipamount]
&&+ \int_{\mb{R}} \gamma(t,X(t^-),u(t^-),z) \tilde{N}(dt,dz),
\end{array}
\end{equation}
\end{samepage}

\noindent and define
\[
J^u_{\mu}(x) := E^x[\int_0 ^T f(t,X(t),u(t)) dt + g(X(T)) + \mu M(X(T))].
\]

We also define the set of stochastic multipliers by

\[
 \Lambda := \{ \mu : \Omega \rightarrow \mb{R}\mbox{ } | \mbox{ } \mu \mbox{ is } \mc{F}_T\mbox{-measurable and } E[\mu] < \infty\}.
\]

Now, we will prove Theorem~\ref{thm: 2}, which says that if we can find a solution to the unconstrained problem~\eqref{eq: LagrangeProblem2} with a stochastic multiplier which ensures that the constraint $M(X(T)) = 0$ a.s. is satisfied, then we have a corresponding solution to our original problem~\eqref{eq: Problem2}.

\begin{theorem}
 \label{thm: 2}
Suppose that we for all $\mu \in \Lambda$ can find $\phi_{\mu}(x)$ and $u_{\mu}^*$ solving the unconstrained stochastic control problem~(\ref{eq: LagrangeProblem2}). Moreover, suppose there exists $\mu_0 \in \Lambda$ such that
\[
M(X_{u_{\mu_0}^*}(T)) =0 \mbox{ a.s.}
\]
Then, $\phi(x) := \phi_{\mu_0}(x)$ and $u^* := u^*_{\mu_0}$ solves the constrained stochastic control problem~(\ref{eq: Problem2}).

\end{theorem}

\begin{proof}
Let $\mu$ be $\mc{F}_T$-measurable. Then,

\[
\begin{array}{lllll}
 E^x[\int_0^T f(t,u_{\mu}^*, X_{u_{\mu}^*})dt + g(X_{u_{\mu}^*}(T)) + \mu M(X_{u_{\mu}^*}(T))] = J_{\mu}^{u_{\mu}^*}(x) \\[\smallskipamount]
\geq J^{u}_{\mu}(x) = E^x[\int_0^T f(t,u, X_{u})dt + g(X_{u}(T)) + \mu M(X_{u}(T))]
\end{array}
\]
\noindent where the first equality uses the definition of $J^u_\mu$, the second uses the definition of $u_{\mu}^*$ and the final equality uses the definition of $J^u_\mu$.

In particular, if $\mu = \mu_0$ a.s. and $u$ is feasible in the constrained control problem~(\ref{eq: Problem2}), then

\begin{equation}
\label{eq: boks}
M(X_{u_{\mu_0}^*}(T))=0=M(X_{u}(T)) \mbox{ a.s.}
\end{equation}
\noindent from the definition of $\mu_0$ and the assumption that $u$ is feasible in problem~(\ref{eq: Problem2}).

Hence,
\[
\begin{array}{lll}
J^{u^*_{\mu_0}}_{\mu_0}(x) = E^x[\int_0^T f(t,u_{\mu_0}^*, X_{u_{\mu_0}^*})dt + g(X_{u_{\mu_0}^*}(T)) + \mu_0 M(X_{u_{\mu_0}^*}(T))] \\[\smallskipamount]
\geq J^{u}_{\mu_0}(x) =  E^x[\int_0^T f(t,u, X_{u})dt + g(X_{u}(T)) + \mu_0 M(X_{u}(T))]
\end{array}
\]

But $M(X_{u_{\mu_0}^*}(T))=0=M(X_{u}(T)) \mbox{ a.s.}$ from equation~(\ref{eq: boks}), so

\[
J^{u^*_{\mu_0}}(x) = J^{u^*_{\mu_0}}_{\mu_0}(x) \geq J^{u}_{\mu_0}(x) = J^{u}(x)
\]
\noindent for all stochastic controls $u$ feasible in the constrained problem~(\ref{eq: Problem2}). Note that $u^*_{\mu_0}$ is feasible in problem~(\ref{eq: Problem2}), therefore it is an optimal control for this problem.
\end{proof}

\medskip

Note that problem~(\ref{eq: LagrangeProblem2}) is a stochastic optimal control problem of the form in {\O}ksendal and Sulem~\cite{OS-Levy}, with $f_{\mu}=f$ and $g_{\mu}(x)=g(x) + \mu M(x)$. Therefore, we may use some known methods of stochastic control, for example the stochastic maximum principle, to solve the problem. Clearly, $f_{\mu}$ and $g_{\mu}$ are continuous functions. However, in order to use the stochastic maximum principle, we also need $g_{\mu}(x)$ to be a concave function. If this is the case, we can solve this problem using the maximum principle for jump processes, Theorem~\ref{thm: SufficientMaxPrinc}.

It is irrelevant for this solution strategy whether the unconstrained stochastic control problem coming from the stochastic Lagrange multiplier method is solved using the maximum principle (Theorem~\ref{thm: SufficientMaxPrinc}), or some other method of stochastic control. If it is more suitable for the problem, the dynamic programming/Hamilton-Jacobi-Bellman approach to stochastic control of jump diffusions can also be used, see {\O}ksendal and Sulem~\cite{OS-Levy} Theorem 3.1. For the dynamic programming approach, the concavity of the function $g_{\mu} = g + \mu M$ is not necessary. However, the problem must have a Markovian structure.

\begin{remark}

Note that the only thing that distinguishes problem~(\ref{eq: LagrangeProblem2}) from an unconstrained version of problem~(\ref{eq: Problem2}) (i.e. without the constraint $M(X(T))=0$ a.s.) is the $g$-function in the objective function. Hence, when we apply the stochastic maximum principle, Theorem~\ref{thm: SufficientMaxPrinc}, the Hamiltonian is equal for these two problems. Therefore, also the control maximizing the Hamiltonian will be equal. The only difference is the terminal condition of the BSDE for the adjoint processes $p, q, r$, equation~(\ref{eq: AdjointJump}). However, this altered terminal condition clearly affects the solution of the adjoint BSDE, and hence can also affect the optimal control process.

\end{remark}

\begin{remark}
It is important that the constraint of the stochastic optimal control problem depends on $X(T)$, i.e. the state process at the \emph{terminal time}. If we had a constraint of the type $M(X(\tilde{t})) = 0$ a.s., where $\tilde{t}<T$, then the new stochastic control problem coming from the Lagrange multiplier method would not fit the setting of {\O}ksendal and Sulem~\cite{OS-Levy}. This complicates matters significantly.
\end{remark}

\section{The economic model: Optimal consumption with L{\'e}vy wage}
\label{sec: Model}

In this section, we derive an economic model which we will use to analyze wage- and inflation risk in an optimal consumption problem with a terminal constraint. We also give interpretations of the two different types of terminal constraints $(i)$ and $(ii)$ (see Section~\ref{sec: StochLagrange}). In the following Section~\ref{sec: Solution}, we will apply the stochastic multiplier method of Section~\ref{sec: StochLagrange} in order to solve this optimal consumption problem. 

As before, let $(\Omega, \mc{F}, P)$ be a probability space, where $\Omega$ is the scenario space, $\mc{F}$ a respective $\sigma$-algebra and $P$ the probability measure. Consider an agent who is planning for times $t \in [a,T+a]$, $a \geq 0$. Again, let $\{B(t,\omega)\}_{t \in [a,T+a]}$ be a Brownian motion, and $\int_{\mb{R}} \gamma_i(t, z,\omega) {N_i}(dt,dz), i=1,2$, two pure jump processes independent of $B(t)$ and each other (we will often abbreviate by omitting $\omega$ in the notation). Let $\{\mc{F}_t\}_{t \in [a,T+a]}$ be the filtration generated by the Brownian motion and the pure jump processes, including all null sets. In the following, by adapted processes, we mean adapted with respect to this filtration.

An agent in this market receives an exogenous nominal wage, chooses a consumption function and has the possibility to save or borrow (i.e., go long or short) in an asset with risk free nominal payoff. For a specific agent, the starting time of planning is $a\geq 0$, $W_{n}(t,\omega)$ is the nominal wage rate at time $t$ and $X_{n}(t,\omega)$ is the nominal level of savings at time $t$. In this economy, $\xi(t,\omega)$ is the time $t$ price of the consumption good, and $S (t)$ is the time $t$ price level of the risk free asset. Our goal is to study the effect of inflation- and wage risk. Hence, to avoid unnecessary complication, we do not include any risky assets. For more about the economic terms, see Romer~\cite{Romer}.

The inflation, denoted $\pi (t,\omega)$, can be decomposed into a drift term $\hat{\pi}$, and the deviation from that level. The deviation from $\hat{\pi}$ is stochastic, and modeled by $\Delta(t,\omega)dt=\tilde{\pi}dB(t)$ ($\tilde{\pi}$ is constant), i.e., $\pi (t,\omega)=\hat{\pi}+\Delta(t,\omega)$. Since inflation is defined by the identity
\[d\xi(t,\omega):=\xi(t,\omega)\pi (t,\omega)dt,\]
the price $\xi$ consequently is a geometric Brownian motion. The market is normalized by setting $\xi (0):=1$.

The changes in the nominal wage rate is a pure jump process
\[dW_n (t) =\xi(t)\int_{\mb{R}} z N_{1}(dt,dz)-(1-\epsilon) W_n(t)\int_{\mb{R}}N_{2}(dt,dz). \]
Here, $N_i(dt,dz)$, $i=1,2$, represent two independent Poisson random measures, $z>0$ is the size of each jump, which will always be positive, and $\epsilon \in (0,1)$ is a constant. Thus, we model the nominal wage process such that an agent receives positive wage gains, and bears the risk of losing a portion of wage (for instance by becoming unemployed, but receiving a welfare benefit). The wage gain pressure is proportional to the consumption price, since an agent cares about the real wage, rather than the nominal wage. Using the L\'{e}vy-Kintchine representation, these terms are decomposed into martingale and non-martingale expressions. In order to do this decomposition, we assume that $\int_{|z| \geq 1} |z| \nu_i(dz) < \infty$, $i=1,2$ (alternatively, the stronger assumption of finite variance of the two L\'{e}vy processes). Thus,
\[
\begin{array}{lll}
\xi(t) \int_{\mb{R}} z {N_1}(dt,dz) &=& \xi(t)\int_{\mb{R}} z (E[N_1(dt,dz)]+\tilde{N_1}(dt,dz)) \\[\smallskipamount]
&=& \xi(t)\left(\int_{\mb{R}} z \nu_1 (dz)dt+\int_{\mb{R}} z \tilde{N_1}(dt,dz)\right) \\[\smallskipamount]
&=& \xi(t)\left(\alpha dt+\int_{\mb{R}} z \tilde{N_1}(dt,dz)\right),
\end{array}
\]
\noindent where $\alpha := \int_{\mb{R}} z \nu_1 (dz)$. Similarly
\[(1-\epsilon)W_n(t)\int_{\mb{R}} {N_2}(dt,dz):=\beta(1-\epsilon)W_n(t) dt+(1-\epsilon)W_n(t)\int_{\mb{R}} \tilde{N_2}(dt,dz).\]
Here, $\beta:=\int_{\mb{R}}\nu_2(dz)$. The terms
\[
\tilde{N_i}(dt, dz) := N_i(dt, dz) - \nu_i(dz)dt,\; i=1,2,
\]
are the compensated Poisson random measures, and $\nu_i(dz)$, $i= 1,2$, are the L\'{e}vy measures. For more on these measures and L\'{e}vy processes in general, see {\O}ksendal and Sulem~\cite{OS-Levy}.

We let the consumption good (with price $\xi(t)$) be the num\'{e}raire for all $ \mbox{ } t \in [a,T+a]$, and define the real wage process by
\[ W(t,\omega):=\frac{W_{n}(t,\omega)}{\xi(t,\omega)}.\]

Then, $P$-a.e.
\[ dW(t,\omega)=d\left(\frac{W_{n}(t,\omega)}{\xi(t,\omega)}\right)=\frac{dW_{n}(t,\omega)}{\xi(t,\omega)}-\frac{W_{n}(t,\omega)d\xi(t,\omega)}{[\xi(t,\omega)]^{2}}.\]

Inserting the relevant terms, we get

\[
\begin{array}{lll}
dW(t) &=& (\alpha - [\hat{\pi}+\beta (1-\epsilon)] W(t))dt - \tilde{\pi}W(t)dB(t)+\int_{\mb{R}} z \tilde{N}_{1}(dt,dz)\\[\smallskipamount]
&&-(1-\epsilon) W(t)\int_{\mb{R}}\tilde{N}_{2}(dt,dz).
\end{array}
\]

Thus, the real wage is a process with both jumps and continuous movements due to the inflation. Figure \ref{fig: Figure1} is an illustration of such a process.

\begin{figure}
\centering
\includegraphics[scale=.6]{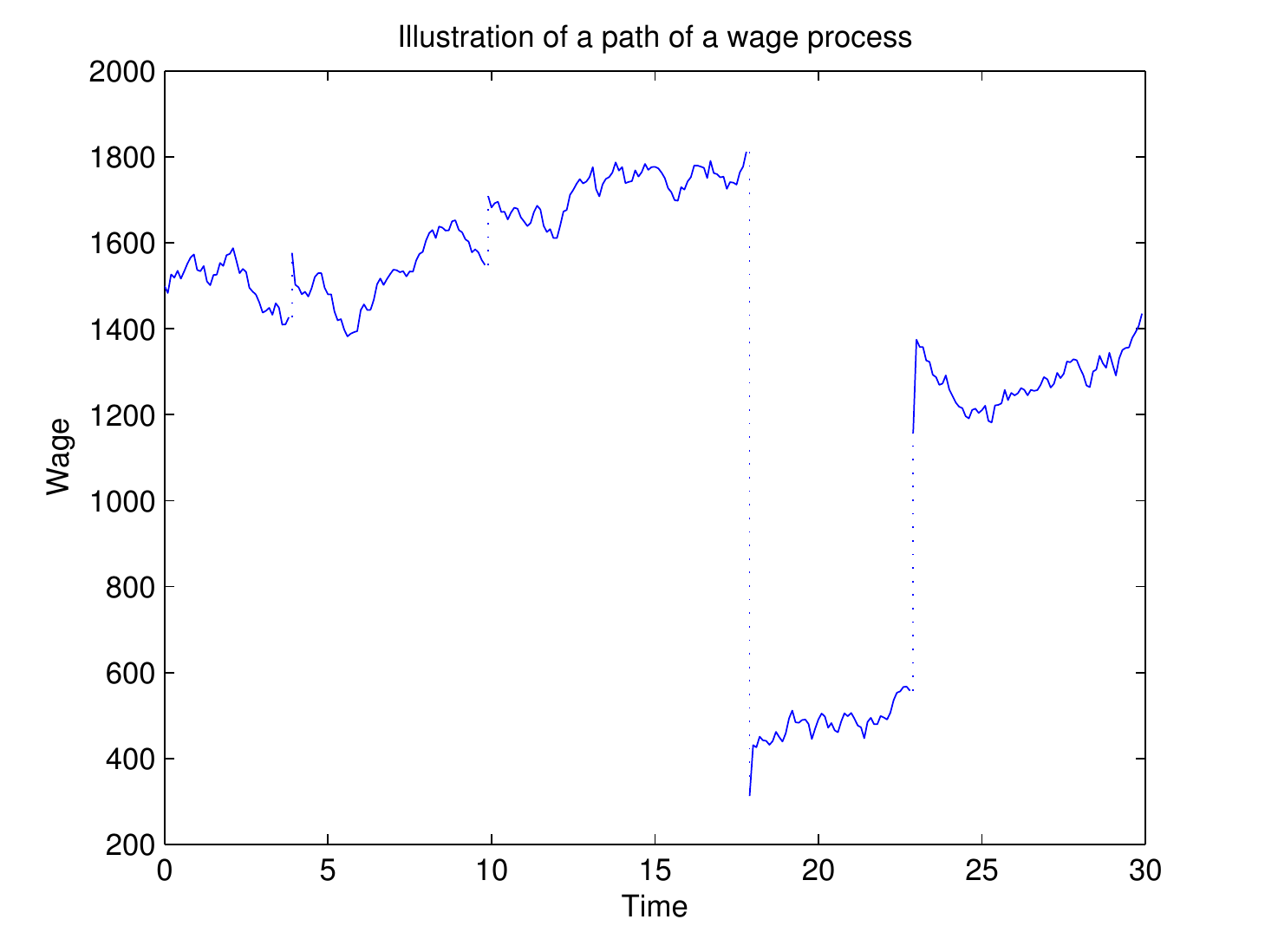}
\caption{A path of a L{\'e}vy wage process.}
\label{fig: Figure1}
\end{figure}

The price of the asset is given by
\[dS(t)=r_nS(t)dt,\]
where the interest rate $r_n$ is a constant. Since this is the only financial object of this market, nominal savings is given by $X_n(t)=\eta S(t)$, where $\eta$ is the number of assets held by the agent. Nominal savings is assumed to be $\Gamma$-financing, meaning that $dX_n(t)=(W_n(t)-\xi(t)c(t))dt +r_n X_n(t)dt$ (this corresponds to the budget constraint), where $\{c(t)\}_{t \in [a,T+a]}$ is an adapted stochastic process representing the real consumption rate of the agent at time $t$.

Define the real value of savings
\[ X(t,\omega):=\frac{X_{n}(t,\omega)}{\xi(t,\omega)}.\]
Written in differential form (using It{\^o}'s formula), we get
\[
\begin{array}{lll}
 dX(t,\omega)=d\left(\frac{X_{n}(t,\omega)}{\xi(t,\omega)}\right)=\frac{dX_{n}(t,\omega)}{\xi(t,\omega)}-X(t,\omega)\pi(t,\omega) \\[\medskipamount]
\hspace{2cm} =\frac{r_{n}X_{n}(t,\omega)dt+W_{n}(t,\omega)dt-\xi(t,\omega)c(t)dt}{\xi(t,\omega)}-X(t,\omega)\pi(t,\omega)dt.\\
\end{array}
\]

Thus, the real value of savings is

\[dX(t,\omega)= (r_{n}-\hat{\pi})X(t)dt - \tilde{\pi}X(t)dB(t)+ W(t)dt - c(t)dt.\]


Given the market situation just described, an agent planning consumption and saving at a given time $a\geq 0$ would like to solve the following stochastic optimal control problem:

\begin{equation}
\begin{array}{lllll}
\label{eq: Agent}
 \sup_{\{c(t) \geq 0 \mbox{ } \forall \mbox{ } t \mbox{ } a.s.\}} \hspace{1cm} E[\int_a^{T+a} e^{-\delta (t-a)} u(c(t)) dt] \hspace{3cm}
 \\[\smallskipamount]
\mbox{subject to}
\end{array}
\end{equation}
\[
\begin{array}{rll}
dX(t) &=& (r_{n}-\hat{\pi})X(t)dt - \tilde{\pi}X(t)dB(t)+ W(t)dt - c(t)dt, \\[\smallskipamount]
dW(t) &=& (\alpha - [\hat{\pi}+\beta (1-\epsilon)] W(t))dt - \tilde{\pi}W(t)dB(t) \\[\smallskipamount]
&&+\int_{\mb{R}} z \tilde{N}_{1}(dt,dz)-(1-\epsilon) W(t)\int_{\mb{R}}\tilde{N}_{2}(dt,dz)\\[\smallskipamount]
X(a) &=& x_a  , \mbox{ } W(a) = w_a \\[\smallskipamount]
\mbox{(i) }E[X(T+a)] \mbox{ }&\geq& \mbox{ } K, \;\;\;\;\;\;\;  or \;\;\;\;\;\;\; \mbox{(ii) }  X(T+a) \geq K \mbox{ a.s.},
\end{array}
\]
\noindent where $K \leq 0$ is a given constant, and either condition (i) or condition (ii) holds. We call this problem (OCP) (optimal consumption problem).

 In problem ~\ref{eq: Agent}, the function $u(\cdot): \mb{R}_+ \mapsto \mb{R}_+$ is the utility function of the agent, which we later will assume is of CRRA form, i.e. $u(c) = \frac{c^\gamma}{\gamma}$, where $\gamma < 1$. Furthermore, $\delta>0$ is the agent's time preference discount factor and, as mentioned, $\{X(t)\}_{t\in [a,T+a]}$ is the stochastic saving process for the time $t$ amount of real wealth placed in the bank. Moreover, $\{W(t)\}_{t \in [a,T+a]}$ is the stochastic wage process for the time $t$ real wage level. We assume that $X(a)=x_a>0$ and $W(a)=w_a\geq 0$. These two initial levels are both exogenous. Note also that the wage process $W$ is given exogenously, while the agent can control the process $X$.

\medskip

\begin{remark}
 \label{remark: M}
Note that the constraints in problem~(\ref{eq: Agent}) correspond to choosing $M(x) = x - K$ in the notation of Section~\ref{sec: StochLagrange}, where $K$ is a constant.
\end{remark}

\medskip

Regarding the two versions of the terminal condition, condition (i) is very mild for the agent. In this case the bank takes on all the market risk. On the other hand, condition (ii) is a strict constraint from the agent's point of view, since it says that the agent must end up with greater than or equal $K$ at the final time a.s. When considering this constraint, the agent bears all market risk. Reality is most likely somewhere in-between the two extremes (i) and (ii). In general, one could introduce a risk sharing parameter $\Lambda \in [0,1]$, and introduce a constraint of the form $\Lambda E[X(T+a)] + (1-\Lambda) X(T+a) \geq K$ a.s. Such a parameter $\Lambda$ characterizes the power of the banks in the market. Note that $\Lambda=1$ is constraint (i), while $\Lambda=0$ gives constraint (ii). Hence, the smaller $\Lambda$ is, the greater is the power of the banks in the market. See Figure~\ref{fig: BankPower} for an illustration.
%
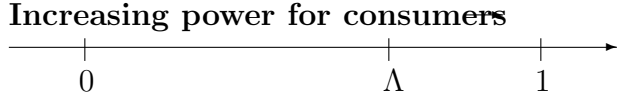
\begin{figure}
\setlength{\unitlength}{1cm}
\begin{picture}(5,2.6)
\put(1,1){\vector(1,0){8}} 
\put(2,0.82){\line(0,1){0.3}} 
\put(1.92,0.42){$0$} 
\put(8,0.82){\line(0,1){0.3}} 
\put(7.92,0.42){$1$} 
\put(6,0.82){\line(0,1){0.3}} 
\put(5.92,0.42){$\Lambda$} 
\put(1,1.3){\textbf{Increasing power for consumers}} 
\put(7, 1.43){\vector(1,0){0.5}} 
\end{picture}
\caption{Illustration of the two constraints}
\label{fig: BankPower}
\end{figure}

The objective is now to solve the stochastic optimal control problem~(\ref{eq: Agent}), i.e. find an expression for the optimal consumption process given the dynamics and the terminal condition on savings.

\section{Stochastic multiplier approach and (OCP)}
\label{sec: Solution}



In this section, we return to our original problem~(\ref{eq: Agent}) with constraint $(i)$. This problem will be solved using the techniques of Section~\ref{subsec: Expected}. 

Problem~(\ref{eq: Agent}) with constraint $(i)$ is equivalent to the problem where the constraint is binding, i.e. $E[X(T+a)] = K$. To prove this, assume we have an optimal control $\tilde{c}$ for problem~(\ref{eq: Agent}) where $E[X_{\tilde{c}}(T+a)] > K$. We will show that there exists an $\epsilon > 0$ such that $c := \tilde{c} + \epsilon$ satisfies $E[X_{c}(T+a)] = K$ (i.e. it is a feasible control in problem~(\ref{eq: Agent})). By Lemma~\ref{lemma: SolutionXFinal2} in the Appendix and the definition of $c$,

\[
 X_c(t) = X_{\tilde{c}}(t) - \epsilon F,
\]

\noindent where $F$ is a known, positive random variable composed of exponentials and depending only on the parameters of the model, see equation~(\ref{eq: SDE-Lemma-Solution2}). Hence,

\[
 E[X_{c}(T+a)] = E[X_{\tilde{c}}(T+a)] - \epsilon E[F].
\]

Let $A := E[X_{\tilde{c}}(T+a)] - K > 0$ (by assumption). Then, $ E[X_{c}(T+a)] -K = A - \epsilon E[F]$. We would like to choose $\epsilon$ such that $ E[X_{c}(T+a)] -K = 0$, that is $A - \epsilon E[F] = 0$, i.e.

\[
 \epsilon = \frac{A}{E[F]} > 0 \mbox{ (positive since $E[F] > 0$)}.
\]

Thus, the consumption process $c = \tilde{c} + \epsilon$ is feasible in problem~(\ref{eq: Agent}). By the definition of $c$, $J(c) > J(\tilde{c})$, so since $c$ is feasible, $\tilde{c}$ cannot have been an optimal control. 

Thus, the problem becomes

\[
\begin{array}{lrlll}
\label{eq: AgentSlutt2}
 \sup_{\{c(t) \geq 0 \mbox{ } \forall \mbox{ } t \mbox{ } a.s.\}} &&E[\int_a^{T+a} e^{-\delta (t-a)} u(c(t)) dt] \hspace{4.5cm}
 \\[\smallskipamount]
\mbox{subject to}
\end{array}
\]
\[
\begin{array}{rlll}
dX(t) &=& (r_{n}-\hat{\pi})X(t)dt - \tilde{\pi}X(t)dB(t)+ W(t)dt - c(t)dt \\[\smallskipamount]
dW(t) &=& (\alpha - [\hat{\pi}+\beta (1-\epsilon)] W(t))dt - \tilde{\pi}W(t)dB(t) \\[\smallskipamount]
&&+\int_{\mb{R}} z \tilde{N}_{1}(dt,dz)-(1-\epsilon) W(t)\int_{\mb{R}}\tilde{N}_{2}(dt,dz)\\[\smallskipamount]
E[X(T+a)]&=&K, \mbox{ } X(a)=x_a, \mbox{ } W(a)=w_a.
\end{array}
\]

\smallskip

We rewrite this problem as an unconstrained two-dimensional stochastic control problem using the stochastic multiplier method of Section~\ref{subsec: Expected},

\[
 \begin{array}{lll}
\sup_{\{c(t) \geq 0 \mbox{ } \forall \mbox{ } t \mbox{ } a.s.\}}  \hspace{1cm} E[\int_a^{T+a} e^{-\delta (t-a)} u(c(t)) dt + \lambda (X(T+a)-K)] \hspace{1cm} \\[\medskipamount]
\mbox{subject to} \\[\smallskipamount]
\end{array}
\]
\begin{equation}
\label{eq: FinalProblem2}
\begin{array}{llllll}
dY(t) &=&
\left[ \begin{array}{ccc} (r_n - \hat{\pi})X(t) + W(t) - c(t) \\
\alpha -[\hat{\pi}+\beta (1-\epsilon)]W(t) \end{array} \right] dt
+ \left[ \begin{array}{ccc} -\tilde{\pi}X(t)
\\ -\tilde{\pi} W(t) \end{array} \right] dB(t) \\[\medskipamount]
&&+ \int_{\mb{R}} \left[ \begin{array}{ccc} 0 & 0
\\z & -(1-\epsilon)W(t) \end{array} \right]
\left[ \begin{array}{ccc} \tilde{N}_{1}(dt,dz)\\
\tilde{N}_{2}(dt,dz) \end{array} \right] \\[\medskipamount]
Y(a) &=& \left[ \begin{array}{ccc} x_a \\ w_a \end{array} \right]
\end{array}
\end{equation}
\smallskip

\noindent where $Y(t):= \left(X(t), W(t)\right)^T$, and $\lambda \in \mb{R}$ is a Lagrange multiplier introduced to handle the constraint $E[X(T+a)]=K$.

We solve problem~(\ref{eq: FinalProblem2}) using the stochastic maximum principle for jump diffusions from {\O}ksendal and Sulem~\cite{OS-Levy}, Theorem~\ref{thm: SufficientMaxPrinc}. Note that this theorem is easily generalized to our setting, where we start at time $t=a$ instead of $t=0$. We apply a subjective current value version of the Hamiltonian function, i.e. $\tilde{H}=H e^{\delta(t-a)}$. All the corresponding adjoint functions will be marked with a tilde to emphasize this change.

In this case, the subjective current value Hamiltonian function  is
\[
\begin{array}{llll}
 \tilde{H}(t,y,c,\tilde{p},\tilde{q},\tilde{r}) &=u(c) + \tilde{p}_1(\{r_n - \hat{\pi}\}x + w - c) + \tilde{p}_2(\alpha - [\hat{\pi}+\beta (1-\epsilon)]w) \\[\smallskipamount]
&- \tilde{q}_1 \tilde{\pi}x - \tilde{q}_2 \tilde{\pi}w + \int_{\mb{R}} z \tilde{r}_{2 1}(t,z) \nu_{1}(dz_{1})-\int_{\mb{R}}(1-\epsilon)w\tilde{r}_{22}(t,z)\nu_{2}(dz_{2})
\end{array}
\]
\noindent where  the adjoint processes $\tilde{p}(t):=(\tilde{p}_1(t),\tilde{p}_2(t))^{T}$, $\tilde{q}(t):=(\tilde{q}_1(t),\tilde{q}_2(t))^{T}$ and $\tilde{r} \in \mb{R}^{2 \times 2}$ is the matrix with components $\tilde{r}_{i j}(t,z)$, $i,j \in \{1,2\}$ for $t \in [a,T+a]$, $z \in \mb{R}$.

The set of adjoint backward stochastic differential equations (BSDEs) corresponding to this Hamiltonian is

\begin{equation}
\label{eq: AdjointFinal2}
\begin{array}{rlll}
d\tilde{p}_1(t) &=& -\tilde{p}_1(t)(r_n - \hat{\pi} - \delta)dt + \tilde{q}_1(t)\tilde{\pi}dt + \tilde{q}_1(t)dB(t) \\[\smallskipamount]
&&+  \int_{\mb{R}} \tilde{r}_{1 1}(t,z_1)\tilde{N}_1(dt,dz_1) + \int_{\mb{R}} \tilde{r}_{1 2}(t,z_2)\tilde{N}_2(dt,dz_2) \\[\medskipamount]
\tilde{p}_1(T+a) &=& \tilde{\lambda}
\end{array}
\end{equation}
\[
\begin{array}{rlll}
d\tilde{p}_2(t) &=& \{-\tilde{p}_1(t) + [\hat{\pi}+\beta (1-\epsilon) + \delta]\tilde{p}_2(t) \\[\smallskipamount]
&&+ \tilde{q}_2(t)\tilde{\pi} + \int_{\mb{R}}(1-\epsilon)\tilde{r}_{22}\nu_{2}(dz_2)\} dt + \tilde{q}_2(t)dB(t) \\[\smallskipamount]
&& + \int_{\mb{R}} \tilde{r}_{2 1}(t,z_1)\tilde{N}_1(dt,dz_1) + \int_{\mb{R}} \tilde{r}_{2 2}(t,z_2)\tilde{N}_2(dt,dz_2)\\[\medskipamount]
\tilde{p}_2(T+a) &=& 0.
 \end{array}
\]

\noindent The first order condition for the maximization of the Hamiltonian is:

\[
\label{eq: OptConsumFinal2}
\tilde{p_1}(t)=u'(\hat{c}(t))=\hat{c}^{(\gamma -1)}(t),
\]

\noindent where the final equality holds for CRRA utility.

To determine the optimal consumption, it suffices to find the adjoint process $\tilde{p}_1(t)$. This process is possibly stochastic due to the randomness in wage and inflation. Now, since $\tilde{\lambda}$ is constant, $\tilde{p}_1(t)$ has a deterministic end value.  Note that $\tilde{p}_1(t) = \tilde{\lambda} \exp(\{r_n - \hat{\pi} - \delta\}(T+a -t))$, $\tilde{q}_1(t) =0$, $\tilde{r}_{1 1}(t,z) = \tilde{r}_{1 2}(t,z) = 0$ for $t \in [a,T+a]$, $z \in \mb{R}$ is the solution of equation~(\ref{eq: AdjointFinal2}). So, $\tilde{p}_1(t)$ is deterministic. Thus, equivalently to equation~(\ref{eq: OptConsumFinal2}), we have

\begin{equation}
\label{eq: FOClambda2}
\tilde{\lambda}e^{(r_n - \hat{\pi}-\delta)(T+a-t)}=\tilde{p}_1(t)=u'(\hat{c}(t))=\hat{c}^{(\gamma -1)}(t).
\end{equation}

To determine the Lagrange multiplier $\tilde{\lambda}^*$ such that $E[X_{\hat{c}}(T)] = K$, we solve the stochastic differential equation (SDE) for the savings process $X(t)$:
\begin{equation}
\begin{array}{lll}
 \label{eq: SDE-X-Final2}
dX(t) &= \big( (W(t) - \hat{c}(t)) + (r_n - \hat{\pi})X(t)\big) dt - \tilde{\pi}X(t)dB(t) \\[\smallskipamount]
X(a) &= x_a.
\end{array}
\end{equation}

From Lemma~\ref{lemma: SolutionXFinal2}, the solution of equation~(\ref{eq: SDE-X-Final2}) is
\begin{equation}
\label{eq: X-Final-Solution}
\begin{array}{llll}
X(t) = x_ae^{R(t)-R(a)} + \int_a^t e^{R(t)-R(s)} (W(s) - \hat{c}(s)) ds,
\end{array}
\end{equation}
\noindent where we for notational convenience define the stochastic processes
\begin{equation}
\label{eq: kOGh}
\begin{array}{lll}
R(t):=& r_n t - \Pi(t), \\[\smallskipamount]
\Pi(t):= &\hat{\pi}t + \frac{1}{2}\tilde{\pi}^{2}t+\tilde{\pi}B(t).
\end{array}
\end{equation}


We also need to solve the stochastic differential equation for the wage process. This solution is given by the following lemma.

\begin{lemma}
 \label{lemma: WageSDE}
The solution of the stochastic differential equation
\begin{equation}
\label{eq: SDEWageEqn}
\begin{array}{lll}
dW(t) &=& (\alpha-[\hat{\pi}+\beta (1-\epsilon)] W(t))dt -\tilde{\pi}W(t)dB(t)+ \int_{\mb{R}} z \tilde{N}_1(dt,dz) \\[\smallskipamount]
&&-(1-\epsilon)W(t)\int_{\mb{R}} \tilde{N}_2(dt,dz)\\[\smallskipamount]
W(a)&=& w_a
\end{array}
\end{equation}
\noindent is
\begin{equation}
\label{eq: SDEWage}
\begin{array}{lllll}
\vspace{0.8cm} 
W(t) &=& w_a e^{-(\zeta(t)-\zeta(a))}+\int_{a}^{t}\alpha  e^{-(\zeta(t)-\zeta(s))}ds +\int_{a}^{t}\int_{\mb{R}}z e^{-(\zeta(t)-\zeta(s))}\tilde{N}_1(ds,dz),\\
\end{array}
\end{equation}
where
\[ \zeta(t):=\Pi(t) + \beta(1-\epsilon)t-\ln(\epsilon)\int_a^t\int_{\mb{R}}\tilde{N}_2(du,dz). \]

\end{lemma}

\noindent For the proof of Lemma~\ref{lemma: WageSDE}, see the Appendix.

The $\Pi(t)$ occurring in $\zeta(t)$ in Lemma~\ref{lemma: WageSDE} adjusts for changes in the num\'{e}raire consumption good, while the term $\beta(1-\epsilon)t-\ln(\epsilon)\int_a^t\int_{\mb{R}}\tilde{N}_2(du,dz)$ represents the geometric effect of unemployment risk.

By inserting the wage expression~(\ref{eq: SDEWage}) into $X_{\hat{c}}(T+a)$, we find

\[
\begin{array}{lll}
 X_{\hat{c}}(T+a) &=  x_a e^{R(T+a)-R(a)}  +  \int_0^{T} w_a e^{R(T+a)-R(t+a)} e^{- (\zeta(t+a)-\zeta(a))}dt \\[\smallskipamount]
&+ \int_0^{T} \left(\int_0^{t} \alpha e^{R(T+a)-R(t+a)} e^{-(\zeta(t+a)-\zeta(s+a))}ds\right)dt \mbox{ }\\[\smallskipamount]
&+\int_0^{T}\left( \int_0^t \int_{\mb{R}} ze^{R(T+a)-R(t+a)} e^{-(\zeta(t+a)-\zeta(s+a))} \tilde{N}_1(ds,dz)\right)dt  \\[\smallskipamount]
& -\int_0^{T} \tilde{\lambda}^{\frac{1}{\gamma -1}} e^{R(T+a)-R(t+a)} e^{\frac{(r_n - \hat{\pi} - \delta)(T-t)}{\gamma -1}}dt.
\end{array}
\]

\noindent Now, we would like to choose $\tilde{\lambda} \in \mb{R}$ such that $E[X_{\hat{c}}(T+a)]=K$. Since $\tilde{\lambda}$ is a constant, it is separable from the integral containing it. The equation can thus be solved, and we find

\[
\tilde{\lambda}^* =  (\hat{\mathcal{B}}^{-1} \; \hat{\mathcal{C}})^{(\gamma - 1)},
\]

where
\[
\begin{array}{rllll}
 \hat{\mathcal{C}} &:=& x_a e^{\hat{r}T} + \int_0^{T} w_{a}E[e^{R(T+a) - R(t+a)-(\zeta(t+a)-\zeta(a))}]dt \\
\hspace{1cm}\\
&& + \int_0^{T} \left(\int_a^t \alpha E[e^{R(T+a)-R(t+a)-(\zeta(t+a)-\zeta(s+a))}] ds\right)dt - K, \\[\smallskipamount]

\hat{r}&:=&r_n - \hat{\pi}, \\[\smallskipamount]

\hat{\mathcal{B}} &:=& \left( \int_0^{T} e^{(\hat{r}-\hat{\Gamma})(T-t)}dt \right)=\frac{e^{(\hat{r}-\hat{\Gamma})T} - 1}{\hat{r} - \hat{\Gamma}},
\end{array}
\]

and $\hat{\Gamma}:=\frac{\hat{r}-\delta}{1-\gamma}$.

Intuitively, $\hat{\mathcal{C}}$ is the expectation of the time $T+a$ future value of consumption, and $\hat{\mathcal{B}}^{-1}$ is a time $T+a$ future value risk aversion weight of the expected subjective real interest rate.

From this, the candidate for the optimal consumption process is

\begin{equation}
\label{eq: OptConsumption4}
 \hat{c}(t) =\hat{\mathcal{B}}^{-1}\cdot \hat{\mathcal{C}}\cdot  e^{-\hat{\Gamma} (T + a - t)}, \mbox{ }t \in [a,T+a].
\end{equation}

Now, we would like to apply the stochastic maximum principle, see the Appendix Theorem~\ref{thm: SufficientMaxPrinc}, to conclude that $\hat{c}(t)$ actually is the optimal consumption process. In order to do this, the boundedness conditions of Theorem~\ref{thm: SufficientMaxPrinc} must hold. However, this is the case for our problem. The functions $\sigma$ and $\gamma$ (for our specific problem) are continuous, and the Brownian motion $B(t)$ has a continuous version. Since the conditions of Theorem~\ref{thm: SufficientMaxPrinc} involve the expectation of the integral over a finite time interval, the integral of these processes will be finite. The adjoint processes $\tilde{p}$ and $\tilde{q}=\tilde{r}=0$ do not cause any finiteness-problems, hence, we may use the stochastic maximum principle.

Therefore, by the stochastic maximum principle, Theorem~\ref{thm: SufficientMaxPrinc}, the consumption process $\hat{c}$ is optimal if it is feasible, i.e. if $\hat{c}(t) \geq 0$ for all $t \in [a,T+a]$. Clearly, since $K \leq 0$, this holds. Therefore, $c^*(t) := \hat{c}(t)$, $t \in [a,T+a]$ is the optimal consumption process.

Note that when we are using a soft end constraint, there are no stochastic elements in the resulting optimal control. The result is that $\lambda$ is a constant, thus the shadow price of the constraint $K$ is constant. Furthermore the adjonint function $\tilde{p}_1(t)$ also becomes deterministic, making  the optimal consumption process deterministic. With the soft end constraint (type $(i)$), the consumer makes a consumption plan at time $a$ only expecting to reach the terminal level $K$. Then, independently of which outcomes of wage and inflation are realized, the consumer adjusts savings, keeping the same consumption path. The result is that the bank bears all market risk, and the value function of the consumer is independent of which states are realized.

To conclude, the soft end constraint implies that the consumer is unaffected by the uncertainty in inflation and wage level, in the sense that only expected total income and expected real interest rate are relevant for the consumption plan, and thus the value function.
%

\begin{remark}
 \label{remark: constraint_(ii)}
The version of problem~(\ref{eq: Agent}) with a type $(ii)$ constraint (i.e. an almost sure constraint on terminal savings) is more complicated. The technique is the same as before, except for determining the Lagrange multiplier $\tilde{\mu}^*$. Now, we would like to find $\tilde{\mu}^*$ s.t. $X_{\hat{c}}(T+a) = K$ a.s. To do this, we insert the expressions for $W$ and $\hat{c}$ into $X(T+a)$. Due to the stochastic nature of $\tilde{\mu}$, the Lagrange multiplier is no longer separable from the integral expression containing it, and hence it is more difficult to determine. However, the adjoint variables have the same interpretations as before. By this we deduce that the stochastic Lagrange multiplier has the ordinary shadow price interpretation, but in this case (with a hard end constraint) the shadow price of $K$ is a random variable, as the marginal utility of changing $K$ will depend on which states are realized.

The value function depends on the changes in the inflation and wage level, and thus it is stochastic. The consumer with the a.s. constraint bears all the market risk. This implies that consumption is strongly influenced by the uncertainty in the inflation and wage level, and in this case the lender bears no risk at all. The possibility of a loss of wage or an increased inflation has a negative effect for a power utility consumer, even though the consumer may be better off than in the soft end constraint case.

\end{remark}

\section{Concluding remarks}

In this paper we have derived a stochastic Lagrange multiplier method, and showed how this can be combined with the stochastic maximum principle for jump diffusions to solve constrained stochastic optimal control problems. As an application, we have studied an optimal consumption problem with inflation- and wage risk.


We have observed that the Lagrange multiplier behaves very differently, depending on whether the end constraint holds in expectation (type $(i)$ constraint) or almost surely (type $(ii)$). This affects the value function of the problem. In the first case, the value function is deterministic and almost identical to a version of the problem without uncertainty (see Syds{\ae}ther et al.~\cite{Sydseter} for such a problem). In the second case, the value function depends on the changes in the inflation and wage level, and thus it is stochastic. By economic intuition, this is equivalent to a consumer bearing no risk in the first case, and all risk in the second case. Hence, the behavior of the consumer is unaffected by inflation and wage risk when there is a terminal constraint which needs to hold in expectation, while the behavior is strongly affected by the risk when there is an a.s. end constraint.

The fact that different terminal constraints influence the outcome of the stochastic optimal consumption problem should be taken into consideration when stochastic optimal control is applied for analysis of practical problems, especially if the stochastic components have large variations, or when worst case scenarios are of interest.

\medskip

\textbf{Acknowledgments:} We are also very grateful for the help of our supervisors, Professor Bernt {\O}ksendal and Professor Kjell Arne Brekke, both at the University of Oslo.

\section{Some results from stochastic analysis}
\label{sec: Appendix}
Let $f : \mb{R}_+ \times \mb{R} \times \mb{R} \rightarrow \mb{R}$ and $g: \mb{R} \rightarrow \mb{R}$ be given, continuous functions. Consider the stochastic optimal control problem
\begin{equation}
 \label{eq: OpprProblem2}
\begin{array}{lll}
 \sup_{u \in \mc{A}} &E[\int_0 ^T f(t,X(t),u(t)) dt + g(X(T))] \\[\smallskipamount]
\mbox{subject to} \\[\smallskipamount]
&dX(t) = b(t,X(t),u(t))dt + \sigma(t,X(t),u(t))dB(t) \\[\smallskipamount]
&\hspace{1.35cm} + \int_{\mb{R}} \gamma(t,X(t^-),u(t^-),z) \tilde{N}(dt,dz)\\[\smallskipamount]
&X(0)=x
\end{array}
\end{equation}
\noindent where $\mc{U} \subset \mb{R}$ is a given set, $b: \mb{R}_+ \times \mb{R} \times \mc{U} \rightarrow \mb{R}$, $\sigma: \mb{R}_+ \times \mb{R} \times \mc{U} \rightarrow \mb{R}$ and $\gamma: \mb{R}_+ \times \mb{R} \times \mc{U} \times \mb{R} \rightarrow \mb{R}$. The control $u(t) = u(t,\omega): \mb{R}_+ \times \Omega \rightarrow \mc{U}$ is admissible, denoted $u \in \mc{A}$, if the dynamics of $X$ has a unique, strong solution for all $x \in \mb{R}$ and $E^x[\int_0 ^T f(t,X(t),u(t)) dt + g(X(T))] < \infty$.

In the following theorem, the function $H : [0,T] \times \mb{R} \times \mc{U} \times \mb{R} \times \mb{R} \times \mc{R} \rightarrow \mb{R}$ is the Hamiltonian function, defined by
\begin{equation}
 \label{eq: Hamiltonian}
\hspace{0.5cm} H(t,x,u,p,q,r) = f(t,x,u) + b(t,x,u)p + \sigma(t,x,u)q + \int_{\mb{R}} \gamma(t,x,u,z)\nu(dz)
\end{equation}
\noindent where $\mc{R}$ is the set of functions such that the integral above converges.
\begin{theorem}
 \label{thm: SufficientMaxPrinc}
(A sufficient maximum principle for stochastic optimal control with jumps, Framstad et al.~\cite{FOS})

\smallskip

Let $\hat{u}$ be an admissible control, i.e. $\hat{u} \in \mc{A}$, with corresponding state process $\hat{Y} = Y^{\hat{u}}$ and suppose there exists a solution $(\hat{p}(t), \hat{q}(t), \hat{r}(t,z))$ of the corresponding adjoint equation

\begin{equation}
\begin{array}{llll}
 \label{eq: AdjointJump}
dp(t)= &- \nabla_y H(t, Y(t), u(t), p(t), q(t), r(t,\cdot))dt \\[\smallskipamount]
&+ q(t)dB(t) + \int_{\mb{R}} r(t^-, z)\tilde{N}(dt,dz), &t<T, \\[\smallskipamount]
p(T) =& \nabla g(Y(T))
\end{array}
\end{equation}
\noindent satisfying
\[
 E \big[ \int_0^T(\hat{Y}(t)- Y^{u}(t))^2 \big\{ \hat{q}^2(t) + \int_{\mb{R}} r^2(t,z)\nu(dz) \big\} dt \big] < \infty
\]
\noindent and
\[
\begin{array}{lll}
E \big[ \hat{p}^2(t) \big\{ \sigma^2(t, Y^u(t),u(t)) + \int_{\mb{R}} \gamma^2(t,Y^u(t),u(t),z) \nu(dz) \big\} dt \big] \\[\smallskipamount]
\hspace{1cm} < \infty \mbox{ for all } u \in \mc{A}.
\end{array}
\]

Moreover, suppose that
\[
 H(t,\hat{Y}(t), \hat{u}(t), \hat{p}(t), \hat{q}(t), \hat{r}(t,\cdot)) = \sup_{v \in \mc{U}} H(t,\hat{Y}(t), v, \hat{p}(t), \hat{q}(t), \hat{r}(t,\cdot))
\]
\noindent for all $t$, that $g$ is a concave function of $y$ and that
\[
 \hat{H}(y) := \max_{v \in \mc{U}} H(t,y,v,\hat{p}(t), \hat{q}(t), \hat{r}(t,\cdot))
\]
\noindent exists and is a concave function of $y$, for all $t \in [0,T]$. Then, $\hat{u}$ is an optimal control.
\end{theorem}

\smallskip

\begin{proof}
 See Framstad et al.~\cite{FOS}.
\end{proof}

\medskip

\begin{lemma}
 \label{lemma: SolutionXFinal2}
Consider the stochastic differential equation,
\begin{equation}
\label{eq: SDE-Lemma2}
\begin{array}{llll}
  dX(t) &= \left( \mu(t) + \alpha X(t)   \right) dt + \beta X(t)dB(t) \\[\smallskipamount]
X(a) &= x_a
\end{array}
\end{equation}
\noindent where $\mu(t)$ is an adapted stochastic process and $\alpha, \beta \in \mb{R}$.

Then,
\begin{equation}
 \label{eq: SDE-Lemma-Solution2}
\begin{array}{lll}
X(t) &=& \exp(\alpha t - \frac{1}{2}\beta^2 t + \beta B(t)) \big( x_a\exp\{-\big(\alpha a - \frac{1}{2}\beta^2 a + \beta B(a)\big)\} \\[\smallskipamount]
&&+ \int_a^t \exp(-\alpha s + \frac{1}{2} \beta^2 s - \beta B(s))\mu(s)ds \big).
\end{array}
\end{equation}

\end{lemma}

\begin{proof}
 The idea is to get rid of the terms of the SDE involving $X(t)$ by multiplying with the integrating factor

\[
 J(t) :=\exp\{-\big(\alpha t + \frac{1}{2}\beta^2 t - \beta B(t)\big)\}.
\]

\noindent By It\^{o}'s product rule (see Exercise 4.3 i {\O}ksendal~\cite{Oksendal}),

\[
 d(X(t)J(t)) = X(t)dJ(t) + J(t)dX(t) + dX(t)dJ(t).
\]

\noindent It\^{o}'s formula implies that
\[
 dJ(t)= \left( (-\alpha + \beta^2)dt -\beta dB(t) \right) \exp(\alpha t + \frac{1}{2}\beta^2 t - \beta B(t)).
\]

\noindent Hence,
\[
 d(X(t)J(t)) = \exp (\alpha t + \frac{1}{2}\beta^2 t - \beta B(t)) \mu(t) dt.
\]

\noindent So, by integrating from $a$ to $t$ on both sides and multiplying by $\frac{1}{J(t)}$, we find the solution

\[
\begin{array}{lll}
X(t) &=& \exp(\alpha t - \frac{1}{2}\beta^2 t + \beta B(t)) \big( x_a\exp\{-\big(\alpha a - \frac{1}{2}\beta^2 a + \beta B(a)\big)\} \\[\smallskipamount]
&&+ \int_a^t \exp(-\alpha s + \frac{1}{2} \beta^2 s - \beta B(s))\mu(s)ds \big).
\end{array}
\]
This concludes the proof.
\end{proof}

\medskip

\begin{lemma}
\label{lemma: BSDEVetIkke}
(Solution of linear BSDE, Proposition 1.3, El Karoui et al.~\cite{ElKarouiEtAl}) Consider a linear BSDE of the form
\[
 \begin{array}{llll}
  -dY(t) &= (\phi(t) + Y(t)\beta(t) + Z(t)\mu(t))dt - Z(t)dB(t) \\[\smallskipamount]
Y(\bar{T}) &= \xi(\bar{T})
 \end{array}
\]
\noindent where $\xi(\bar{T})$ is an $\mc{F}_{\bar{T}}$-measurable random variable. This BSDE has a unique solution $(Y,Z)$, where $Y$ is explicitly given by
\[
 Y(t) = E[\xi(\bar{T})\Gamma_{t,\bar{T}} + \int_t^{\bar{T}} \Gamma_{t,s} \phi(s)ds | \mc{F}_t]
\]
\noindent where
\[
\begin{array}{llll}
 d\Gamma_{t,s} &= \Gamma_{t,s}(\beta(s)ds + \mu(s)dB(s)) \\[\smallskipamount]
\Gamma_{t,t} &=1, \mbox{ } \Gamma_{t,s} \Gamma_{s,u}= \Gamma_{t,u} \mbox{ } \forall \mbox{ } t \leq s \leq u.
\end{array}
\]

\end{lemma}

\begin{proof}
 See El Karoui et al.~\cite{ElKarouiEtAl}.
\end{proof}

Finally, we give the proof of Lemma~\ref{lemma: WageSDE}.

\begin{proof}(Proof of Lemma~\ref{lemma: WageSDE})
 We solve the SDE by multiplying with the integrating factor
\[
J(t)=\exp \left([\hat{\pi}+\beta(1-\epsilon)]t+\tilde{\pi}B(t)+\frac{1}{2}\tilde{\pi}^{2}t-\ln(\epsilon)\int_a^t\int_{\mb{R}}\tilde{N}_2(du,dz)\right),
\]
\noindent (chosen to get rid of the terms involving $W$ in equation~(\ref{eq: SDEWageEqn})). By the It\^{o} product rule for jump processes (see {\O}ksendal and Sulem~\cite{OS-Levy}, Exercise 1.2),
\[
\begin{array}{lllll}
d(W(t)J(t))&=&J(t)\big(\alpha -[\hat{\pi}+\beta(1-\epsilon)]W(t)\big)dt-J(t)\tilde{\pi}W(t)dB(t) \\[\smallskipamount]
&&+W(t)J(t)\big(\hat{\pi}+\beta (1-\epsilon)+\tilde{\pi}^{2}\big)dt+W(t)J(t)\tilde{\pi}dB(t) \\[\smallskipamount]
&&-\tilde{\pi}^{2}W(t)J(t)dt+\int_{\mb{R}}J(t-)z\tilde{N}_1(ds,dz)\\[\smallskipamount]
&&+\int_{\mb{R}}\big(-(1-\epsilon)W(t)J(t)(\frac{1}{\epsilon}-1)+W(t)J(t)(\frac{1}{\epsilon}-1)\\[\smallskipamount]
&&-W(t)J(t)(1-\epsilon)\big)\tilde{N}_2(dt,dz)\\[\smallskipamount]
&=&\alpha J(t) dt + \int_{\mb{R}}zJ(t-)\tilde{N}_1(dt,dz).\\[\smallskipamount]

\end{array}
\]

So,
\[
W(t)J(t)=w_aJ(a)+ \int_{a}^{t}\alpha J(s) ds +\int_{a}^{t}\int_{\mb{R}}zJ(s-)\tilde{N}_1(ds,dz),
\]
which gives the solution in equation~(\ref{eq: SDEWage}).
\end{proof}

\end{document}